\documentclass[12pt]{amsart}

\usepackage{tgpagella}
\usepackage{euler}
\usepackage[T1]{fontenc}
\usepackage{amsmath, amssymb}
\usepackage[hidelinks]{hyperref}
\usepackage[english]{babel}
\usepackage{mathrsfs}
\usepackage{eucal}
\usepackage[all]{xy}
\usepackage{tikz}

\newtheorem{thm}{Theorem}[section]
\newtheorem*{thm*}{Theorem}
\newtheorem{lem}[thm]{Lemma}
\newtheorem{fact}[thm]{Fact}

\newtheorem{prop}[thm]{Proposition}
\newtheorem*{prop*}{Proposition}

\newtheorem{cor}[thm]{Corollary}
\newtheorem*{cor*}{Corollary}

\theoremstyle{definition}
\newtheorem{defn}[thm]{Definition}
\newtheorem*{defn*}{Definition}

\newtheorem{example}[thm]{Example}

\newtheorem*{question*}{Question}
\newtheorem*{Pquestion*}{Popa's question}

\newtheorem*{conv*}{Convention}

\newcommand{\dminus}{ 
\buildrel\textstyle\ .\over{\hbox{ 
\vrule height3pt depth0pt width0pt}{\smash-} 
}}

\def\bb{\mathbb}

\def\bb{\mathbb}

\def\cal{\mathcal}

\def\u{\mathsf 1}

\newcommand{\cstar}{$\mathrm{C}^*$}

\makeatletter

\def\dotminussym#1#2{%
  \setbox0=\hbox{$\m@th#1-$}%
  \kern.5\wd0%
  \hbox to 0pt{\hss\hbox{$\m@th#1-$}\hss}%
  \raise.6\ht0\hbox to 0pt{\hss$\m@th#1.$\hss}%
  \kern.5\wd0}

\newcommand\cU{{\cal U}}

\def \Th{\operatorname{Th}}
\def \R{\mathcal R}
\def \u{\mathcal U}

\begin{document}

%%%%%%%%%%%%%%%%%%%%%%%%%%%%%%%%%%%%%%%%%%%%%%

\title{The undecidability of having the QWEP}
\author{Jananan Arulseelan, Isaac Goldbring, and Bradd Hart}

\address{Department of Mathematics and Statistics, McMaster University, 1280 Main St., Hamilton ON, Canada L8S 4K1}
\email{arulseej@mcmaster.ca}
\email{hartb@mcmaster.ca}
\urladdr{http://ms.mcmaster.ca/~bradd/}

%\address{Department of Mathematics and Statistics, McMaster University, 1280 Main St., Hamilton ON, Canada L8S 4K1}

\address{Department of Mathematics\\University of California, Irvine, 340 Rowland Hall (Bldg.\# 400),
Irvine, CA 92697-3875}
\email{isaac@math.uci.edu}
\urladdr{http://www.math.uci.edu/~isaac}
\thanks{Goldbring was partially supported by NSF grant DMS-2054477.}

\maketitle

\begin{abstract}
We show that neither the class of \cstar-algebras with Kirchberg's QWEP property nor the class of W*-probability spaces with the QWEP property are effectively axiomatizable (in the appropriate languages).  The latter result follows from a more general result, namely that the hyperfinite III$_1$ factor does not have a computable universal theory in the language of W*-probability spaces.  We also prove that the Powers' factors $\R_\lambda$, for $0<\lambda<1$, when equipped with their canonical Powers' states, do not have computable universal theory.  Our results allow us to conclude the existence of a family of \cstar-algebras (resp. a family of W*-probability spaces), none of which have QWEP, but for which some ultraproduct of the family does have QWEP.     
\end{abstract}

\section{Introduction}

Recall that a \cstar-algebra $A\subseteq B(H)$ has the \textbf{weak expectation property} (\textbf{WEP} for short) if there is a ucp map $\Phi:B(H)\to A^{**}$ that is the identity on $A$, while $A$ has the \textbf{QWEP}  if $A$ is isomorphic to a quotient of a \cstar-algebra with the WEP.  \textbf{Kirchberg's QWEP problem} asked whether or not every separable \cstar-algebra has QWEP.  In \cite{kirchberg} (see also \cite{Oz03}), where Kirchberg raises this problem, he also shows that it is equivalent to the Connes Embedding Problem (CEP).  By the recent landmark result in quantum complexity theory known as $\operatorname{MIP}^*=\operatorname{RE}$ \cite{mip}\footnote{Our proofs rely on the fact that the universal theory of R, in the language of tracial von Neumann algebras, is not computable as described in \cite{GH}.  In turn, this result currently relies on the coding of Turing machines described in the paper MIP*=RE, \cite{mip}}, the QWEP problem is now known to have a negative answer.

In \cite{goldbring}, the second author showed that the class of \cstar-algebras with QWEP forms an elementary class in the first-order language of \cstar-algebras.  Now that it has been established that this class forms a proper subclass of the class of all \cstar-algebras, one may ask how different these classes are from one another.  In this paper, we show that, from the perspective of computability theory, they are wildly different.  Indeed, while the class of all \cstar-algebras admits an effectively enumerable axiomatization, the first main result of this paper is that the same cannot be said for the subclass of \cstar-algebras with QWEP:

\begin{thm*}
There is no effectively enumerable set of sentences in the language of \cstar-algebras whose models are precisely the \cstar-algebras with QWEP.
\end{thm*}

This theorem will follow from a much more general result appearing as Theorem \ref{mainqwep} below.  In \cite{GH}, the second and third authors used  $\operatorname{MIP}^*=\operatorname{RE}$ to prove that the universal theory of the hyperfinite II$_1$ factor $\R$ is not computable.  In fact, they proved that there cannot exist any effectively enumerable set of sentences in the language of tracial von Neumann algebras that are true in $\R$ and all of whose models embed into an ultrapower of $\R$.  Since a finite von Neumann algebra $M$ has QWEP if and only if it embeds into an ultrapower of $\R$ (and in a way which preserves any given faithful normal trace on $M$), it follows that there can be no effective axiomatization of the finite QWEP von Neumann algebras.  This latter fact, along with other techniques used in \cite{GH}, is what allows one to prove the previous theorem.

It is worth remarking that the proof given in \cite{goldbring} that the class of \cstar-algebras with QWEP is elementary was soft and simply showed that the class of QWEP algebras was closed under ultraproducts and ultraroots.  The absence of concrete axioms is thus explained by the previous theorem.

We can also use the previous theorem to prove the following fact about ultraproducts of \cstar-algebras without the QWEP property:

\begin{thm*}
There is a family $(A_i)_{i\in I}$ of \cstar-algebras without the QWEP which have an ultraproduct $\prod_\u A_i$ with the QWEP.
\end{thm*}

Our method of proving the previous theorem works equally well in the case of tracial von Neumann algebras (using the main result of \cite{GH}), allowing us to conclude:

\begin{thm*}
There is a family $(M_i)_{i\in I}$ of tracial von Neumann algebras that do not embed into an ultrapower of the hyperfinite II$_1$ factor $\R$ which have a (tracial) ultraproduct $\prod_\u M_i$ that does embed into an ultrapower of $\R$.
\end{thm*}

Our next results have us return to the setting of von Neumann algebras, but enlarge our perspective from the class of finite von Neumann algebras to the class of $\sigma$-finite von Neumann algebras.  By \cite[Theorems 3.4 and 4.2]{AHW}, due to Ando, Haagerup, and Winslow, a separably acting von Neumann algebra $M$ has the QWEP if and only if it embeds into the Ocneanu ultrapower $\R_\infty^\u$ of the \textbf{hyperfinite type III$_1$ factor} $\R_\infty$ with expectation.   By a \textbf{W*-probability space}, we mean a pair $(M,\varphi)$ consisting of a $\sigma$-finite von Neumann algebra equipped with a distinguished faithful, normal state.  Embeddings of W*-probability spaces are $*$-homomorphisms which admit state-preserving expectations.  The model-theoretic approach to studying W*-probability spaces was initiated by Dabrowski in \cite{dabrowski} and further developed in work of the second author and Houdayer in \cite{GHo}.  In particular, in the latter paper it was observed that whenever $M$ is a type III$_1$ factor, then all W*-probability spaces $(M,\varphi)$ have the same theory.  Consequently, given any W*-probability space $(M,\varphi)$, we see that $M$ is QWEP if and only if there is an embedding $(M,\varphi)\hookrightarrow (\R_\infty,\psi)^\u$ of W*-probability spaces, where $\psi$ is \emph{any} faithful, normal state on $\R_\infty$.  Thus, the common universal theory of structures of the form $(\R_\infty,\psi)$ axiomatizes the class of QWEP W*-probability spaces.  A particular consequence of Theorem \ref{infinityEP} below is the following:

\begin{thm*}
There is no effectively enumerable set of sentences in the language of W*-probability spaces that axiomatizes precisely the class of QWEP W*-probability spaces.
\end{thm*}

In particular, the universal theory of $\R_\infty$ is not computable.

The preceding theorem has the following consequence, which is the W*-probability space version of our second theorem above:

\begin{thm*}
There is a family $(M_i)_{i\in I}$ of W*-probability spaces without the QWEP which have an ultraproduct $\prod_\u M_i$ with the QWEP.
\end{thm*}

The Ando-Haagerup-Winslow result referred to above also holds when replacing $\R_\infty$ by any Powers factors $\R_\lambda$ (where $0<\lambda<1$), which is the unique hyperfinite type III$_\lambda$ factor.  However, in this case, as one varies the states on $\R_\lambda$, one no longer has a unique universal theory.  Nevertheless, for the canonical Powers state $\varphi_\lambda$ on $\R_\lambda$, one can still prove an undecidability result, which is Theorem \ref{powers} below:

\begin{thm*}
For any $0<\lambda<1$, the universal theory of $(\R_\lambda,\varphi_\lambda)$ is not computable, where $\varphi_\lambda$ is the Powers state on $\R_\lambda$.
\end{thm*}

The key point in proving the previous theorem is that, for a faithful, normal \textbf{lacunary} state $\varphi$ on a von Neumann algebra $M$, the centralizer of $\varphi$, $M_\varphi$, is a tracial von Neumann algebra (when equipped with the restriction of $\varphi$) which is effectively definable in $M$.

While the model theory of type III$_0$ factors is poorly behaved (for example, the class is not stable under ultraproducts), it would still be interesting to determine which separable hyperfinite type III$_0$ factors have uncomputable universal theory.

In order to keep this note relatively short, we only define the notions crucial for understanding the proofs that follow.  In particular, a complete discussion of computability of theories as it pertains to the setting at hand can be found in \cite{GH}.

\section{The undecidability of QWEP}

The following is a more precise version of the first main theorem from the introduction.

% The following theorem is a more precise version of the first theorem appearing in the introduction.  It requires the notion of a \cstar-algebra with the uniform Dixmier property, which we recall here for the sake of the reader.

% \begin{defn}
% Given $m\in \bb N$ and $0<\gamma<1$, we say that a unital \cstar-algebra $A$ has the \textbf{$(m,\gamma)$-uniform Dixmier property} if, for all self-adjoint $a\in A$, there are unitaries $u_1,\ldots,u_m\in U(A)$ and $z\in Z(A)$ such that
% $$\left\|\sum_{i=1}^m \frac{1}{m}u_iau_i^*-z\right\|\leq \gamma\|a\|.$$ We say that $A$ has the \textbf{uniform Dixmier property} if it has the $(m,\gamma)$-Dixmier property for some $m$ and $\gamma$.  
% \end{defn}

% Given $m$ and $\gamma$, let $\theta_{m,\gamma}$ denote the following sentence in the language of \cstar-algebras:
% $$\sup_a\inf_{u_1,\ldots,u_n}\inf_\lambda\max\left(\max_{i=1,\ldots,n}\|u_iu_i^*-1\|,\|\sum_{i=1}^m\frac{1}{m}u_iau_i^*-\lambda\|\dminus \gamma\|a\|\right).$$
% Here, the supremum is over self-adjoint contractions, the first infimum is over contractions, and the second infimum is over the unit disk in $\bb C$.  As mentioned in \cite[Section 6]{GH}, if $A$ is a simple unital \cstar-algebra with the $(m,\gamma)$-uniform Dixmier property, then $\theta_{m,\gamma}^A=0$.  On the other hand, if $\theta_{m,\gamma}^A=0$, then $A$ is monotracial.

\begin{thm}\label{mainqwep}
There is no effectively enumerable theory $T$ in the language of \cstar-algebras with the following two properties:
\begin{enumerate}
    \item All models of $T$ have QWEP.
    \item There is an infinite-dimensional, monotracial, model of $T$ whose unique trace is faithful.
\end{enumerate}
\end{thm}

\begin{proof}
Suppose, towards a contradiction that such $T$ existed.  Take a model $A$ of $T$ as in the second condition in the statement of the theorem and let $\tau_A$ denote the unique faithful trace on $A$.  Work now in the language of tracial \cstar-algebras and consider the theory $T'$ consisting of the axioms for tracial \cstar-algebras together with $T$.  It is clear that $T'$ is effective and $(A,\tau_{A})\models T'$.

Fix a universal sentence $\sigma$ in the language of tracial von Neumann algebras; notice that $\sigma$ can be considered as a special kind of sentence in the language of $T'$ (one that does not mention the operator norm).  

\noindent \textbf{Claim:}  We have that 
\[
\sup\{\sigma^{(B,\tau_B)} \ : \ (B,\tau_B)\models T'\}=\sigma^{(\R,\tau_{\R})}. \quad (\dagger)
\]

\noindent \textbf{Proof of Claim:}  We first show that $\sigma^{(A,\tau_A)}\geq\sigma^{(\R,\tau_{\R})}$, establishing that the left hand side of $(\dagger)$ is at least as big as the right hand side of $(\dagger)$.  Let $M$ denote the von Neumann algebra generated by the image of $A$ with respect to its GNS representation corresponding to the trace $\tau_A$ and let $\tau_M$ be the induced trace on $M$.  In general, one always has that $\sigma^{(M,\tau_M)}\geq \sigma^{(A,\tau_A)}$.  Indeed, even though the GNS representation in general need not be injective, recall that the matrix of $\sigma$, that is, the quantifier-free part of $\sigma$ that appears after the universal quantifiers, is a formula that only mentions the trace and not the operator norm, whence the GNS representation preserves the value of the matrix of $\sigma$ when elements of $A$ are substituted for the variables.  Moreover, since $A$ is SOT dense in $M$, it follows that $\sigma^{(M,\tau_M)}=\sigma^{(A,\tau_A)}$.  
Since $A$ is monotracial, $\tau_M$ is the unique trace on $M$, whence $M$ is a factor. 
Since $\tau_A$ is faithful, $A$ embeds into $M$, and since $A$ is infinite-dimensional, $M$ is consequently a II$_1$ factor.  It follows that $(M,\tau_M)$ contains a copy of $(\R,\tau_{\R})$, whence $\sigma^{(A,\tau_A)}=\sigma^{(M,\tau_M)}\geq \sigma^{(\R,\tau_{\R})}$, as desired.  

On the other hand, fix an arbitrary model $(B,\tau_B)$ of $T'$; we wish to show that $\sigma^{(B,\tau_B)}\leq \sigma^{(\R,\tau_{\R})}$.  Let $N$ denote the von Neumann algebra generated by the image of $B$ with respect to its GNS representation corresponding to the trace $\tau_B$ and let $\tau_N$ be the induced trace on $N$.  As  in the previous paragraph, $\sigma^{(B,\tau_B)}\leq \sigma^{(N,\tau_N)}$.   Since $B$ has QWEP, so does its image under the GNS representation and consequently $N$ also has QWEP (see \cite{kirchberg}).  It follows that $(N,\tau_N)$ admits a trace-preserving embedding into $(\cal R,\tau_R)^\u$ (see \cite[Corollary 6.2]{Oz03}), and thus $\sigma^{(N,\tau_N)}\leq \sigma^{(\R,\tau_{\R})}$.  Combining this latter inequality with $\sigma^{(B,\tau_B)}\leq \sigma^{(N,\tau_N)}$ establishes the desired inequality.  This finishes the proof of the claim.

% The argument now essentially proceeds as in \cite[Theorem 6.4]{GH}.  Indeed, suppose that $(B,\tau_B)\models T'$ and let $N$ denote the von Neumann algebra generated by the image of $B$ in the GNS representation corresponding to $\tau_B$.  Then $N$ is a QWEP von Neumann algebra (since $B$ had QWEP) and is a II$_1$ factor with unique trace $\tau_N$ (since $B$ is monotracial).  Consequently, $N$ admits a trace-preserving embedding into $\R^\u$.  On the other hand, since $B$ is simple, $B$ embeds into $N$ (in a trace-preserving way).  Altogether, for any universal sentence $\sigma$ in the language of tracial \cstar-algebras, we have that $\sigma^{(B,\tau_B)}=\sigma^{(N,\tau_N)}=\sigma^{(\R,\tau_\R)}$.  

By the claim and the completeness theorem, by running proofs from $T'$, we can find computable upper bounds to $\sigma^{(\R,\tau_\R)}$, contradicting the fact that the universal theory of $\R$ is not effectively enumerable.
\end{proof}

\begin{cor}\label{maincor}
There is no effective theory $T$ in the language of \cstar-algebras such that a \cstar-algebra has QWEP if and only if it is a model of $T$.
\end{cor}

We can use the previous corollary to derive an interesting ``non-closure'' statement for the class of \cstar-algebras without QWEP:

\begin{cor}\label{notclosed}
The class of \cstar-algebras without the QWEP is not closed under ultraproducts.
\end{cor}

\begin{proof}
Suppose, towards a contradiction, that the class of \cstar-algebras without the QWEP is closed under ultraproducts.  Let $T_{C^*}$ denote the (effective) theory of \cstar-algebras.  Since the class of \cstar-algebras with QWEP is axiomatizable and the language of \cstar-algebras is separable, there is a sentence $\sigma_{QWEP}$ in the language of \cstar-algebras such that a \cstar-algebra $A$ has QWEP if and only if $\sigma_{QWEP}^A=0$.  Our contradiction assumption implies that there is some $r>0$ such that $\sigma_{QWEP}^A\geq r$ for all \cstar-algebras without the QWEP.  Without loss of generality, $r\in \bb Q$.   Since the language of \cstar-algebras is computable and the set of computable sentences is dense in the set of all sentences (see \cite[Section 2]{GH}), there is a computable sentence $\psi$ such that $d(\sigma_{QWEP},\psi)<\frac{1}{3}$ in the usual metric on formulae.  It then follows that $T_{C^*}\cup\{\psi\dminus \frac{r}{3}\}$ is an effective axiomatization of the class of \cstar-algebras with QWEP, contradicting Corollary \ref{maincor}.
\end{proof}

It would be interesting to know if the previous corollary could be established simply by knowing that the QWEP conjecture fails or even directly from $\operatorname{MIP}^*=\operatorname{RE}$ itself, or if the model-theoretic tools used here are indispensable to this argument.

Using the analog of Corollary \ref{maincor} for the class of tracial von Neumann algebras proven in \cite{GH}, the exact same line of reasoning shows  the following:

\begin{cor}
The class of tracial von Neumann algebras that do not embed in an ultrapower of the hyperfinite II$_1$ factor $\R$ is not closed under ultraproducts.
\end{cor}

We end this section by mentioning one further consequence of Theorem \ref{mainqwep}.  Call a \cstar-algebra \textbf{pseudo-nuclear} if it is a model of the common theory of the class of nuclear \cstar-algebras.  Equivalently, a \cstar-algebra is pseudo-nuclear if it is elementarily equivalent to an ultraproduct of nuclear \cstar-algebras.  In \cite[Section 7.3]{munster}, the problem of finding a ``natural'' characterization of the elementary class of pseudo-nuclear \cstar-algebras is raised (although the term pseudo-nuclear is not used there).  Theorem \ref{maincor} implies that one cannot find such a characterization if by ``natural'' one means ``effective'':

\begin{cor}
The elementary class of pseudo-nuclear \cstar-algebras is not effectively axiomatizable.
\end{cor}

\begin{proof}
This follows immediately from Theorem \ref{mainqwep} together with the fact that pseudo-nuclear \cstar-algebras are QWEP.
\end{proof}

% \begin{proof}
% Suppose, towards a contradiction, that such a $T$ existed.  Let $m$ and $\gamma$ witness that $\cal Q$ has the uniform Dixmier property.  Work in the language of tracial \cstar-algebras and consider the theory $T'$ of tracial \cstar-algebras together with $T$ and the single condition $\theta_{m,\gamma}=0$.  It is clear that $T'$ is effective and $(\cal Q,\tau_{\cal Q})\models T$.  Since all models of $T$ are such that their GNS have QWEP, we get that all GNS completions are $\R^\u$-embeddable.  Now run the completeness theorem argument.
% \end{proof}

% \begin{remark}
% The same proof shows that there can be no effective theory of \cstar-algebras all of whose models have QWEP and which has a model that is infinite-dimensional, simple, admits a trace, and has the uniform Dixmier property.
% \end{remark}

% Even better:  this should work for any effective $T$ all of whose models are QWEP and for which at least one model has a trace $\tau$ for which the corresponding GNS contains $R$.  This uses the fact (should be true) that if $M$ is a finite QWEP vNa, then no matter what trace $\tau$ one has on $M$, there is a trace preserving embedding of $M$ into $R^\u$.

\section{Reminders on languages for W*-probability spaces}

In this section, we recall Dabrowski's languages for studying W*-probability spaces from \cite{dabrowski} and establish some notation to be used in the rest of this paper.  We also introduce the language used to capture W*-probability spaces in \cite{GHS} and compare this language with the Dabrowski language.

Throughout, we will be considering W*-probability spaces $(M,\varphi)$ and will let $\sigma_t^\varphi$ denote the associated modular automorphism group of $M$.  While one normally considers the norm $\|\cdot\|_\varphi^\#$ on $M$ given by $\|x\|_\varphi^\#:=\sqrt{\varphi(x^*x)+\varphi(xx^*)}$, which defines the strong-* topology on $M$, Dabrowski instead works with the norm $\|\cdot\|_\varphi^*$ on $M$ given by $\|x\|_\varphi^*:=\inf_{y\in M}\sqrt{\varphi(y^*y)+\varphi((x-y)(x-y)^*)}$, which has various advantages when trying to axiomatize W*-probability spaces in an appropriate first-order language.

The main obstacle in Dabrowski's approach to W*-probability spaces is the lack of uniform continuity of multiplication with respect to the above norm.  To overcome this, he works with ``smeared versions'' of multiplication.  To explain this construction, first recall that, for any $f\in L^1(\bb R)$, one can define the function $\sigma_f^\varphi:M\to M$ given by $\sigma_f^\varphi(x)=\int_{\bb R}f(t)\sigma_t^\varphi(x)dt$.  

For $K\in \bb N$, Dabrowski considers the \textbf{Fej\'er kernel} $f_{K}\in L^1(\bb R)$ given by
\[
f_K(t)=\frac{K}{2\pi}1_{\{t=0\}}+\frac{1-\cos(Kt)}{\pi Kt^2}1_{\{t\not=0\}}.
\]
He defines $F^\varphi_K = \sigma_{f_K}^\varphi$ and the smeared multiplication maps $m_{K,L}$ by 
\[
m_{K,L}(x,y):=F^\varphi_K(x)\cdot F^\varphi_L(y).
\]

The language for W*-probability spaces we prefer to consider in this paper is the expansion of what Dabrowski refers to as an ``approximately minimal'' language in \cite[Section 1.4]{dabrowski} by function symbols for the modular automorphism group.  We will refer to this language as $L_{W^*}$.  The domains of quantification correspond to operator norm unit balls and the symbols of the language consist of the constant symbols for $0$ and $1$, rational scalar multiplication, ordinary addition, all smeared multiplications $m_{K,L}$, unary function symbols for the modular automorphisms $\sigma_t^\varphi$ (say for rational $t$), and the real and imaginary parts of the state.  The metric on each sort is given by $\|\cdot\|_\varphi^*$.

Before introducing this fairly small language, Dabrowski instead considers a much larger language that we will temporarily denote $L_{W^*}^{\dagger}$.  While, we will not go into the details of this larger language, we mention that it is clearly a computable language.  In this larger language, he writes down an explicit, effectively enumerable theory $T_{W^*}^{\dagger}$ and in \cite[Theorem 8]{dabrowski} it is shown that this theory axiomatizes the class of W*-probability spaces (viewed as $L_{W^*}^{\dagger}$-structures).  Moreover, setting $T_{W^*}$ to be the set of $L_{W^*}$-consequences of $T_{W^*}^{\dagger}$, in \cite[Theorem 14]{dabrowski} it is shown that $T_{W^*}^{\dagger}$ is a definitional expansion of $T_{W^*}$.  The upshot of all of this is that $T_{W^*}$ is then an effectively enumerable $L_{W^*}$-theory that axiomatizes the class of W*-probability spaces when viewed as $L_{W^*}$-structures.

% In the last section, it will be convenient to work with an intermediate language for W*-probability spaces, namely the language obtained from $L_{W^*}$ by adding symbols for the modular automorphisms $\sigma_t^\varphi$ (for rational $t$); we refer to this language as $L_{W^*}^\dagger$.  These symbols are contained in the language $L_{W^*}^{\dagger\dagger}$, whence this really is an intermediate language.

We now discuss an alternative to $L_{W^*}$ as presented in \cite{GHS}.  The approach taken in \cite{GHS} is based on axiomatizing the action of $(M,\varphi)$ on its standard representation.  One says that an element $a \in M$ is bounded if the actions of $a$ and $a^*$ by left and right multiplication on $M$ extend to bounded operators on the GNS representation via $\varphi$.  By a result of Takesaki, the collection of bounded elements in $M$ is dense in the weak operator topology.  Since the modular automorphisms turn out to be definable in this language, we may assume that we also have unary function symbols to name them.  We introduce a language $\cal{L}_{W^*}$ that captures these observations.  $\cal{L}_{W^*}$ differs from $L_{W^*}$ in three significant ways:
\begin{enumerate}
    \item There is a binary function symbol for (full) multiplication and we do not use the symbols introduced for smeared multiplication.
    \item The domains of quantification are over the sets of $K$-bounded elements (as $K$ varies), that is, those elements $a$ where left and right multiplication by $a$ and $a^*$ have operator norm at most $K$ when viewed as operators on the GNS Hilbert space corresponding to $\varphi$.
    \item The metric is obtained from $\| \cdot \|^\#_\varphi$ (rather than $\|\cdot\|_\varphi^*$).
\end{enumerate}
The first and third points seem to be a simplification of the approach taken by Dabrowski.  The price one pays is quantification is now over bounded elements and not the operator norm balls of $M$.  Moreover, while the axioms for $W^*$-probability spaces in this language \emph{without symbols for the modular automorphism group} is easily seen to be effectively axiomatizable, it is not clear to us that the definitional expansion by naming the modular automorphism group remains effectively axiomatizable; this question is the source of work in progress.

Since tracial von Neumann algebras are in particular W*-probability spaces, we can consider them in either of the above languages.  We note that, since in tracial von Neumann algebras the modular theory is trivial and the notion of $K$-bounded element coincides with that of having operator norm at most $K$, both languages simply revert (in an appropriate sense) to the usual language for tracial von Neumann algebras.  However, if $(M,\tau)$ is a tracial von Neumann algebra, then $\|x\|_\tau^\#=\sqrt{2}\|x\|_\tau$ and $\|x\|_\tau^*=\frac{\sqrt{2}}{2}\|x\|_\tau$ for all $x\in M$ (see \cite[Lemma 4]{dabrowski} for the latter calculation), whence this causes a slight mismatch in the evaluation of formulae in either of the W*-probability languages as opposed to their evaluation in the tracial von Neumann algebra language.

\section{Failure of the $\R_\infty$EP}

We recall the following definition from \cite{GH}:

\begin{defn}
If $L$ is a computable language and $A$ is an $L$-structure, by the \textbf{$A$EP} we mean the statement:  there is an effectively enumerable $L$-theory $T$ contained in $\Th(A)$ so that all models of $T$ embed in an ultrapower of $A$.
\end{defn}

In \cite{GH}, it was shown that the $\R$EP is false, thus providing a stronger refutation of the Connes Embedding Problem.  In this section, we work in the languages $L_{W^*}$ and $\cal L_{W^*}$ and show that the $\R_\infty$EP has a negative solution in each of them.  We begin with the former language.
\begin{defn}
If $\theta$ is an $L_{W^*}$-sentence, let $\bar \theta$ denote the sentence in the language of tracial von Neumann algebras obtained by replacing any smeared multiplication map $m_{K,L}$ by actual multiplication and any appearance of $\|\cdot\|_\varphi^*$ by $\frac{\sqrt{2}}{2}\|\cdot \|_\varphi$.
\end{defn}

Note that $\bar \theta$ has the same quantifier complexity as $\theta$ and $\bar \theta ^{(M,\tau)}=\theta^{(M,\tau)}$, where on the left-hand side of the equation we view $(M,\tau)$ as a structure in the language of tracial von Neumann algebras while on the right-hand side of the equation, we view $(M,\tau)$ as an $L_{W^*}$-structure.  It is also clear that the map $\theta\mapsto \bar \theta$ is a computable map.

% We also have the following:

% \begin{lem}\label{smearing}
% There is a computable function $g:\bb R\to \bb R$ such that, for any tracial von Neumann algebra $(M,\tau)$ and any $L_{W^*}$-sentence $\theta$, we have $\bar \theta^{(M,\tau)}=g(\theta^{(M,\tau)})$.
% \end{lem}

\begin{thm}\label{infinityEP}
The $\cal R_{\infty}EP$ is false in the language $L_{W^*}$.
\end{thm}

\begin{proof}
Suppose, towards a contradiction, that the $\cal R_{\infty}EP$ is true as witnessed by the $L_{W^*}$-theory $T\subseteq \Th(\R_\infty)$.  By assumption, we have that $T$ is an effectively enumerable set of sentences.  We now set $$T':=\{\bar \theta \dminus \epsilon \ : \ \epsilon \in \bb Q,\  \theta \text{ is universal, and }T\vdash  \theta\dminus \epsilon\}.$$  Note that $T'$ is an effectively enumerable set of sentences in the language of tracial von Neumann algebras.  Moreover, we have $T'\subseteq \Th(\R)$.  Indeed, since $\R$ embeds in $\R_\infty^\u$ (as $L_{W^*}$-structures), for any universal $L_{W^*}$-sentence $\theta$, we have that $\bar \theta^\R=\theta^\R\leq \theta^{\R_\infty}$.  Moreover, if $(M,\tau)$ is a tracial von Neumann algebra which is a model of $T'$, then it embeds, as a $W^*$-probability space, into a model of $T$, which is QWEP by assumption. It follows that $M$ is QWEP and thus embeds into $\R^\u$ in a trace-preserving way.  Consequently, $T'$ witnesses that the $\R$EP has a positive solution, which is a contradiction.
% .  In particular, this means that, for any $L_{W^*}$-sentence $\theta$, we have computable upper bounds for $\theta^{\R_\infty}$.  Let $T$ be the set of tracial von Neumann algebra sentences of the form $\bar \theta\dminus \epsilon$, where $\theta$ is a universal $L_{W^*}$-sentence and $\epsilon$ is a positive rational number for which $\theta^{\R_\infty}\leq \epsilon$.  By assumption, we have that $T$ is an effectively enumerable set of sentences.  Note that $T\subseteq \Th(\R)$.  Indeed, since $\R$ embeds in $\R_\infty$ (as $L_{W^*}$-structures), we have that $\bar \theta^\R=\theta^\R\leq \theta^{\R_\infty}$.  Moreover, if $(M,\tau)$ is a tracial von Neumann algebra which is a model of $T$, then $\Th_\forall (M,\tau)\leq \Th_\forall(\R_\infty)$ (as $L_{W^*}$-theories) whence $(M,\tau)$ embeds in an ultrapower of $\R_\infty$.  It follows that $M$ is QWEP and thus embeds into $\R^\u$ in a trace-preserving way.  This contradicts the failure of the $\R EP$.
\end{proof}

A particular consequence of the previous theorem is that the universal theory of $\R_\infty$ is not effectively enumerable.  As mentioned in the introduction, a W*-probability space is QWEP if and only if it is a model of the universal theory of $\R_\infty$.  Thus, the second theorem from the introduction follows from Theorem \ref{infinityEP}.  Moreover, arguing just as in the case of Corollary \ref{notclosed}, we have the following:

\begin{cor}
The class of W*-probability spaces without the QWEP is not closed under ultraproducts.
\end{cor}

We now point out that Theorem \ref{infinityEP} is true if one considers $\R_\infty$ in the language $\cal L_{W^*}$.  The only issue now is that quantification might be over sorts of bounded elements; however, if the state in question is a trace, there is no distinction between the operator norm and the left or right bound.  Also, in the de-smearing process, one should replace any appearance of $\|\cdot\|_\varphi^\#$ by $\sqrt{2}\|\cdot\|_\varphi$.  We conclude then by the same proof as for Theorem \ref{infinityEP} that:
\begin{thm}
The $\cal R_{\infty}EP$ is false in the language $\cal L_{W^*}$.
\end{thm}

It is worth pointing out that the failure of the $\cal R_{\infty}EP$ in the language $L_{W^*}$ is indeed a strengthening of the refutation of the CEP since the theory of $W^*$-probability spaces in that language is effectively enumerable.  However, since the same cannot be said for the language $\cal L_{W^*}$, we cannot immediately reach the same conclusion.  However, the proof of Theorem \ref{infinityEP} does show that there cannot be an effectively axiomatizable $\cal L_{W^*}$-theory with only QWEP models.

\section{Failure of the $\R_\lambda$EP}

Our goal in this section is to show that, for any $\lambda\in (0,1)$, the $(\R_\lambda,\varphi_\lambda)EP$ has a negative solution in the languages $L_{W^*}$ and $\cal L_{W^*}$, where $\varphi_\lambda$ is the Powers state on $\R_\lambda$.  Throughout this section, we work in the language $L_{W^*}$.  The proof is essentially identical for $\cal L_{W^*}$ and so we make comments along the way to help the reader. Recall that $T_{W^*}$ is the theory in the language $L_{W^*}$ which axiomatizes the class of $W^*$-probability spaces.  In parallel, let $\cal{T}_{W^*}$ be the $\cal{L}_{W^*}$-theory of the class of $W^*$-probability spaces.

To begin, we need some preparation.

Given a W*-probability space $(M,\varphi)$, the \textbf{centralizer} of $\varphi$ is 
$$M_\varphi:=\{x\in M \ : \ \sigma^\varphi_t(x)=x \text{ for all }t\in \mathbb{R}\}=\{x\in M \ : \ \varphi(xy)=\varphi(yx) \text{ for all }y\in M\}.$$

We note the following obvious facts about the centralizer:

\begin{lem}

\

\begin{enumerate}
    \item $M_\varphi$ is a finite von Neumann algebra with trace $\varphi|M_\varphi$.  
    \item The unit ball of $M_\varphi$ is a zeroset in $(M,\varphi)$, namely the zeroset of the quantifier-free formula $\sum_k 2^{-k}d(\sigma^\varphi_{t_k}(x),x)$, where $(t_k)$ is an enumeration of the rationals.  This is a formula in both languages.
    \item If $x\in M_\varphi$, then $\sigma_f^\varphi(x)=x$ for all $f\in L^1(\bb R)_+$ with $\|f\|_1=1$.
\end{enumerate}
\end{lem}

Recall that the faithful normal state $\varphi$ on $M$ is said to be \textbf{lacunary} if $1$ is an isolated point of the spectrum of the modular operator $\Delta_\varphi$.  

\begin{example}
Suppose that $M$ is a type III$_\lambda$ factor and $\varphi$ is a \textbf{periodic} faithful normal state on $M$ with period $\frac{2\pi}{|\log(\lambda)|}$.  Then $\sigma(\Delta_\varphi)\subseteq \{0\}\cup\lambda^{\mathbb{Z}}$.  In particular, $\varphi$ is a lacunary state on $M$.  Moreover, in this case, by a result of Connes (see \cite[Theorem 4.2.6]{Co72}), $M_\varphi$ is a II$_1$ factor.  In the special case of $\mathcal{R}_\lambda$ with the Powers state $\varphi_\lambda$, we have that $(\mathcal{R}_\lambda)_\varphi$ is the hyperfinite II$_1$ factor $\mathcal{R}$. 
\end{example}

\cite[Proposition 4.27]{AH} states that if $\varphi$ is a lacunary faithful normal state on $M$, then $(M^\u)_{\varphi^\u}=(M_\varphi)^\u$.  In other words:

\begin{prop}
If $\varphi$ is a lacunary faithful normal state on $M$, then the unit ball of $M_\varphi$ is a definable subset of the unit ball of $M$.  This result holds in both languages.
\end{prop}

It turns out that, in the context of the previous proposition, there is a very nice formula that witnesses the definability of $M_\varphi$. In the next definition, we recall that $C^1_b(\bb R)$ denotes the set of continuously differentiable functions $\bb R\to \bb R$ with bounded derivative.

\begin{defn}
Call $f:\bb R\to \bb R$ \textbf{good} if $f\in L^1(\bb R)_+\cap C^1_b(\bb R)$.
\end{defn}

We first note the following easy fact about good functions:

\begin{lem}
Suppose that $f:\bb R\to \bb R$ is computable and good.  Then $f$ is \textbf{effectively $L^1$}, that is, there is an algorithm such that, for any rational $\epsilon>0$, computes $n\in \bb N$ such that $\|1_{[-n,n]^c}f\|_1<\epsilon$.
\end{lem}

\begin{proof}

Suppose that $f$ is good and let $\epsilon>0$ is given. Since good functions are Lipschitz,  they are in particular Riemann integrable. Thus, we can compute effectively the integral of $f$ on the complement of $[-n,n]$. Compute this integral for successive $n$ until you get a value less than $\epsilon$. This halts for any given $epsilon$ because otherwise $f$ has infinite norm. Thus $f$ is effectively $L^1$.
\end{proof}

The following is \cite[Lemma 5]{dabrowski}.
\begin{fact}
If $f:\bb R\to \bb R$ is good, then
$$\|\sigma_f^\varphi(x)-\frac{1}{n^2}\sum_{k=-n^3}^{n^3-1}f(\frac{k}{n^2})\sigma_{k/n^2}^\varphi(x)\|_\varphi^*\leq \|1_{[-n,n]^c}f\|_1\|x\|_\varphi^*+\frac{2\|f\|_1}{n^2}\|x\|_\varphi^\#+\frac{\|f'\|_\infty}{n}\|x\|_\varphi^*.$$
\end{fact}

The previous fact has the following obvious model-theoretic consequence:

\begin{lem}
Suppose that $f$ is good.  Then the function $\sigma_f^\varphi$ is a uniform limit of $L_{W^*}$-terms (so is, in particular, a definable function).  Moreover, since $f$ is effectively $L^1$, then there is an algorithm such that, upon input $n\in \bb N$, returns an $L_{W^*}$-term $t_n$ such that $d(\sigma_f^\varphi(x),t_n(x))<\frac{1}{n}$ for all $x\in M_1$. This Lemma is also true in the language $\cal{L}_{W^*}$.
\end{lem}

\begin{defn}
For $f\in L^1(\bb R)_+$, we say that $f$ is a \textbf{$\lambda$-function} for $0 < \lambda < 1$ if $\|f\|_1=1$ and $\operatorname{supp}(\hat f)\subset (\log (\lambda),-\log(\lambda))$ (where $\hat f$ is the Fourier transform of $f$). 
\end{defn}

If $\varphi$ is a lacunary faithful normal state on $M$, we say that $\varphi$ is \textbf{$\lambda$-lacunary} if $\lambda\in (0,1)$ is such that $\sigma(\Delta_\varphi)\cap (\lambda,\frac{1}{\lambda})=\{1\}$.  (This terminology seems to be nonstandard but convenient.)  In particular, note that if $\varphi$ is a periodic faithful normal state on $M$ with period $\frac{2\pi}{|\log(\lambda)|}$, then $\varphi$ is $\lambda$-lacunary.

The following lemma is included in the proof of \cite[Proposition 4.27]{AH}:

\begin{lem}
Suppose that $\varphi$ is a $\lambda$-lacunary faithful normal state on $M$ and $f$ is a $\lambda$-function. Then for all $x\in M$, $\sigma_f^\varphi(x)$ belongs to $M_\varphi$.
\end{lem}

Call $f:\bb R\to \bb R$ a \textbf{$\lambda$-good function} if it is both good and a $\lambda$-function.  We note the following basic fact about $\lambda$-good functions:

\begin{prop}
$\lambda$-good functions exist.
\end{prop}
\begin{proof}
Without loss of generality, we assume that $\lambda = \frac{1}{e}$. Consider the bump function $\hat{g}$ defined to be $\exp(-\frac{1}{1-x^2})$ on $(-1, 1)$ and $0$ otherwise (scale horizontally by a computable real less than $\lambda$ for the general case). Note that $\hat{g}$ is a computable function. By the computability of Riemann integrals, the inverse Fourier transform $g$ of $\hat{g}$ is also computable. While $\|g\|_1$ may not be 1, we may scale the resulting function $g$ to $f$ with $\|f\|_1=1$ and note that $f$ remains $\lambda$-good.
\end{proof}

The following lemma explains the significance of $\lambda$-good functions:

\begin{lem}
Suppose that $f$ is a $\lambda$-good function and consider the $T_{W^*}$-formula $P_f(x):=d(x,\sigma_f^\varphi(x))$.  Then:
\begin{enumerate}
    \item $P_f(x)$ is a quantifier-free $T_{W^*}$-formula.  Moreover, since $f$ is effectively $L^1$, then there is an algorithm which, upon in put $n\in \bb N$, returns a quantifier-free $L_{W^*}$-formula $\psi_n(x)$ such that $\|P_f-\psi_n\|<\frac{1}{n}$ (in models of $T_{W^*}$).
    \item If $\varphi$ is a $\lambda$-lacunary faithful normal state on $M$, then:
    \begin{enumerate}
        \item the zeroset of $P_f$ in $(M,\varphi)$ is $M_\varphi$, and
        \item for all $x\in M$, we have $d(x,M_\varphi)\leq P_f(x)$.
    \end{enumerate} 
\end{enumerate}
This Lemma is also true in the language $\cal{L}_{W^*}$ relative to the theory $\cal{T}_{W^*}$.
\end{lem}

In the previous section, we ``de-smeared'' formulae to obtain tracial von Neumann algebra formulae.  In this section, we consider the reverse process:
\begin{defn}
Given a tracial von Neumann algebra formula $\theta(x)$, let $\theta^\dagger(x)$ denote the ``smearing'' of $\theta$ obtained by replacing every appearance of multiplication by the smeared multiplication function symbol $m_{1,1}$. We also replace all instances of $\|\cdot\|_\varphi$ by $\sqrt{2}\|\cdot\|^*_\varphi$. We obtain a formula in the language $L_{W^*}$.
\end{defn}

The following lemma is clear; recall that for a formula $\phi$ in $L_{W^*}$, $\bar \phi$ is the de-smearing of $\phi$ introduced in the previous section:

\begin{lem}\label{smear}
There are effective enumerations $(\theta_n)$ and $(\phi_n)$ of the computable tracial von Neumann algebra formulae and $L_{W^*}$-formulae respectively and computable functions $g,h:\bb N\to \bb N$ such that $\theta_n^\dagger=\phi_{g(n)}$ and %$\theta_n=\overline{\theta_{h(n)}^\dagger}$.
$\bar \phi_m = \theta_{h(m)}$.  Moreover, $\overline{\theta_n^\dagger}=\theta_n$ for all $n\in \bb N$.
\end{lem}

% \begin{proof}
% IS THIS OBVIOUS?  DO WE NEED A PROOF? Probably obvious enough that we do not need to write it but here is the proof anyway. Given $n$, we have by assumption a way to enumerate up to $\theta_n$. We can then enumerate $\phi_n$ until until we find $phi_m$ that is equal to $\theta_n$. This halts because both sequences by the assumption that they enumerate the same set (up to finding and replacing multiplication symbols which is finitary and hence computable), $\theta^\dagger_n$ must appear at the enumeration of $\phi$ at some finite m. Set $g(n)=m$.
% \end{proof}

We now come to the main results of this section:

\begin{thm}\label{powers}
For any $\lambda\in (0,1)$, the universal theory of $(\R_\lambda,\varphi)$ is not computable in the language $L_{W^*}$ nor in the language $\cal{L}_{W^*}$, where $\varphi$ is the Powers state on $\cal R_\lambda$.
\end{thm}

\begin{proof}
We show that if the universal theory of $(\R_\lambda,\varphi)$ is computable, then so is the universal theory of $\R$. We first do the proof in the language $L_{W^*}$. To see this, suppose that $\sup_x\theta_m(x)$ is a universal sentence in the language of tracial von Neumann algebras.  (Here, and throughout this proof, all variables range over the unit ball.)  Let $\alpha:[0,1]\to [0,1]$ be a computable, increasing function with $\alpha(0)=0$ satisfying $$|\phi_{g(m)}(x) - \phi_{g(m)}(y)|\leq \alpha(d(x,y))$$ for all $x,y\in M_1$; this is possible (uniformly in $m$) by Lemma \ref{smear} and \cite[Proposition 2.10]{BBHU}.  Fix $\epsilon>0$ and effectively find $n\in \bb N$ such that, if $|r-s|<\frac{1}{n}$, then $|\alpha(r)-\alpha(s)|\leq \epsilon$. (This is possible by the construction of $\alpha$ cited above.) 
% Let $X$ be the definable set which is interpreted as the unit ball of $(R_\lambda)_\varphi$.  

Fix a $\lambda$-good function $f$ and let $\psi_n$ be a quantifier-free $L_{W^*}$-formula such that $\|P_f-\psi_n\|<\frac{1}{n}$.  Consider the universal $L_{W^*}$-sentence 
$$\Phi:\equiv \sup_{z}[\phi_{g(m)}(z)-\alpha(\psi_n(z))].$$
By assumption, we can find an interval $(a,b)\subseteq \bb R$ with $b-a<\epsilon$ and such that $\Phi^{\R_\lambda}\in (a,b)$.  Now suppose that $z\in (\R_\lambda)_\varphi$.  Then $P_f(z)=0$, so $\psi_n(z)<\frac{1}{n}$, whence $\alpha(\psi_n(z))\leq \epsilon$.  But $\phi_{g(m)}(z)-\alpha(\psi_n(z))\leq b$, so 
\[
\sup_{z\in (\cal R_\lambda)_\varphi}(\phi_{g(m)}(z)^{\R_\lambda})\leq b+\epsilon.
\]
On the other hand
\begin{alignat}{2}
\sup_{z\in (\R_\lambda)_\varphi}\left(\phi_{g(m)}(z)\right)^{\R_\lambda}&=\sup_{z\in \R_\lambda}[\phi_{g(m)}(z)^{\R_\lambda}-
\alpha(d(z,(\R_\lambda)_\varphi))]\notag \\ \notag
    &\geq \left(\sup_{z}(\phi_{g(m)}(z)-\alpha(P_f(z)))\right)^{\R_\lambda}.\notag
    \end{alignat}
This last term is greater than or equal to $\Phi^{\R_\lambda}-\epsilon > a-\epsilon$. 
Consequently, 
\[
a-\epsilon<\sup_{z\in (\R_\lambda)_\varphi}\left(\phi_{g(m)}(z)\right)^{\R_\lambda}\leq b+\epsilon.
\]
Since $(b+\epsilon)-(a-\epsilon)<3\epsilon$ and $\overline{\phi_{g(m)}}=\theta_m$, the previous display implies that we can effectively approximate $(\sup_z \theta_m(z))^\R$, which is the desired contradiction.

In the language $\cal{L}_{W^*}$, one can repeat the above argument except there is no need to smear $\theta$.  Since we have full multiplication in this language, we may use $\theta$ as a formula in the language of tracial von Neumann algebras.  However, since the norms are interpreted slightly differently, the obvious analog of Lemma \ref{smear} will be needed to conclude as in the case of $L_{W^*}$.
% Recall that $\R$ is the centralizer of $\R_\lambda$, so
% $$\sup_{x\in \R}\theta^\#(x)=\sup_z[\theta^\#(z)-\alpha(d(z,\R))].$$ Moreover, since the centralizer has a computable modulus of uniform continuity (as argued above, and this does not use the computability of $\lambda$!), as discussed in our earlier paper, this means that for any $\eta$, we can effectively find a computable existential formula $\psi_\eta(z)$ such that $|\psi_\eta(z)-d(z,\R)|<\eta$ for all $z$.  Since $\alpha$ has a computable modulus of uniform continuity (CHECK THAT THIS IS TRUE), given $\delta$, we can compute $\eta=\delta(\epsilon)$ and thus get that $$\left|\sup_{x\in \R}\theta^\#(x)-\sup_z[\theta^\#(z)-\alpha(\psi_\eta(z))]\right|<\epsilon.$$  Now the sentence on the right is a computable universal (perhaps extended) Dabrowski formula, so by assumption we can effectively approximate it, whence we can effectively approximate $\sup_{x\in \R}\theta^\#(x)$.  However, for $x\in \R$, $\theta^\#(x)=\theta(x)$, and we obtain the desired contradiction.
\end{proof}

Since the $(\cal R_\lambda,\varphi_\lambda)EP$ would imply that the universal theory of $(\R_\lambda,\varphi_\lambda)$ is computable, we immediately get:

\begin{cor}
The $(\cal R_\lambda,\varphi_\lambda) EP$ is false in the language $L_{W^*}$ and the language $\cal{L}_{W^*}$.
\end{cor}

The proof of Theorem \ref{powers} shows something more general:

\begin{thm}
Suppose that $(M,\varphi)$ is a W*-probability space such that $M$ is QWEP and $\varphi$ is a lacunary faithful, normal state on $M$ for which $M_\varphi$ contains $\R$ (e.g. if $M_\varphi$ is a II$_1$ factor, which is the case when $\varphi$ is periodic).  Then the universal theory of $(M,\varphi)$ is not computable in the language $L_{W^*}$ and the language $\cal{L}_{W^*}$.
\end{thm}

\begin{proof}
% In Ando-Haagerup they refer to Haagerup-Stormer ``Pointwise inner automorphisms of von Neumann algebras'' to show that in the lacunary case, 

It is well-known that $(M_\varphi,\varphi|M_\varphi)$ embeds in $(M,\varphi)$ as W*-probability spaces.  (For example, one can use that $M_\varphi$ is invariant under the modular automorphism group and apply Takesaki's theorem \cite{Takesaki}.)   Consequently, $M_\varphi$ is also QWEP.  Since the latter is a tracial von Neumann algebra, it embeds in $\R^\u$ in a trace-preserving manner.  On the other hand, we also assumed that $M_\varphi$ contains $\R$, whence $\Th_\forall(\R)=\Th_\forall(M_\varphi)$.  Now we argue as in Theorem \ref{powers} above to reach a contradiction.
\end{proof}

%We leave it to the reader to check that the results in this section also hold for the language $\cal L_{W^*}$.

\end{document}